\documentclass[12pt,a4paper]{amsart}
\usepackage[T2A]{fontenc}
\usepackage{amsmath}
\usepackage{amssymb,amscd}
\usepackage{enumerate}

\usepackage{fullpage}
\usepackage{hyperref}
\usepackage{ulem}
\usepackage{tabularray}
\usepackage{subcaption}
\usepackage{tikz}
\usetikzlibrary{cd,graphs,positioning,arrows,shapes.misc,decorations.pathmorphing}

\theoremstyle{plain}
\newtheorem{theorem}{Theorem}[section]
\newtheorem{lemma}[theorem]{Lemma}
\newtheorem{proposition}[theorem]{Proposition}
\newtheorem{corollary}[theorem]{Corollary}

\theoremstyle{definition}
\newtheorem{definition}[theorem]{Definition}

\newtheorem{notation}[theorem]{Notation}
\newtheorem{remark}[theorem]{Remark}

\theoremstyle{remark}

\def\Aut{\operatorname{Aut}}

\def\GG{{\mathbb G}}

\def\PP{{\mathbb P}}
\def\FF{{\mathbb F}}

\def\QQ{{\mathbb Q}}

\def\NN{{\mathbb N}}
\def\K{{\mathbb K}}

\def\Ga{{\mathbb G}_a}

\renewcommand{\phi}{\varphi}

\makeatletter
\newcommand\thankssymb[1]{\lowercase{\textsuperscript{\@alph{#1}}}}
\makeatother

\address{
HSE University, Faculty of Computer Science\\
Pokrovsky blvd. 11, Moscow, 109028
Russia
}

\email{a@perep.ru}

 \usepackage{xcolor}

\begin{document}

\title[Additive actions on surfaces]{Additive actions on projective surfaces with a finite number of orbits}
\date{\today}
\author{Alexander Perepechko}
\thanks{
  The article was prepared within the framework of the project ``International academic cooperation'' at the HSE University.}

\dedicatory{Dedicated to Professor Leonid Makar-Limanov on the occasion of his 80th birthday}

\begin{abstract} 
An additive action on an algebraic variety is an effective action of the vector group with an open orbit.
We describe projective surfaces with du Val singularities that admit an additive action with a finite number of orbits.
In particular, we provide examples of projective surfaces with 1-parameter families of pairwise non-isomorphic additive actions, which answers the question by Hassett and Tschinkel.
\end{abstract}

\maketitle

\tableofcontents

\section{Introduction}
Equivariant compactifications of the vector group $\Ga^n$ are actively studied in the recent years,
a systematic approach starting with \cite{HaTs99}.
Compactifications with a finite number of orbits are of particular interest, 
see survey \cite{ArZa22} and recent works \cite{BoChSh24,Sh25}.

We consider the case of dimension two and restrict to projective algebraic surfaces with du Val singularities, see \cite{Ar66,DuVal34I}.
The case of singular del Pezzo surfaces over $\QQ$ was studied in \cite{DeLo10}, 
motivated by the Batyrev--Manin conjecture, see also \cite{FrLo19}.
Possible candidate surfaces were narrowed down  mostly by the fact that any curve with negative self-intersection belongs to the complement of the open orbit.
Then the actions were explicitly constructed in homogeneous coordinates.

The compactifications with a finite number of orbits somewhat mimic in this property toric varieties.
In general, complete toric surfaces admit at most two additive actions up to isomorphism, see \cite[Theorem~3]{Dz21}.
The toric surfaces with a finite number of $\GG_a^2$-orbits are Hirzebruch surfaces, weighted projective planes, and $\PP^1\times\PP^1$, see \cite[Proposition~3]{Sh25}.
If we  consider  the hypersurfaces in $\PP^3$,
 then we obtain an additional non-toric compactification, which is a non-degenerate cubic, see \cite[Corollary~4]{BoChSh24}.

Let $X$ be a projective algebraic surface defined over an algebraically closed field $\K$ of characteristic zero.
Assume that $X$ has at most du Val singularities and admits an action of $\Ga^2$ with a finite number of orbits.
We construct all such surfaces in terms of a blowup model, see Theorem~\ref{th:main}.
In particular, we find out the that blowup of a Hirzebruch surface at two points of a fiber admits a one-dimensional family of pairwise non-isomorphic additive actions, see Remark~\ref{rm:moduli}.
This answers positively the question \cite[Section~5.2(2)]{HaTs99}.

In Section~\ref{sec:prelim-models}, we recall some general facts about equivariant completions of $\Ga^2$ and reduce a general case to an equivariant blowup of the Hirzebruch surface $\FF_r$, where $r\ge0$.
In Section~\ref{sec:local-coords}, we introduce local coordinates at any point at the boundary of an equivariant blowup and express $\Ga^2$-actions in them.

In Section~\ref{sec:blowup-models}, we reduce to the only $\Ga^2$-action on $\FF_r$, see Lemma~\ref{lm:contract-to-Fr-phi}, and study the induced action on the exceptional curves of blowups infinitesimally near the distinguished fiber of $\FF_r$, see Lemma~\ref{lm:fixed-curves}.
This allows us to construct resolutions of singularities for $\Ga^2$-surfaces of interest in Proposition~\ref{pr:blowups} and describe dual graphs of their boundaries in Fig.~\ref{fig:blowups}.

In Section~\ref{sec:isom-classes}, we use a criterion whether two equivariant blowups of a $\Ga^2$-surface are equivariantly isomorphic, see Lemma~\ref{lm:norm-isom}, to write down the isomorphism classes of the resolutions of singularities, see \ref{pr:cases-eq-isom}.
We proceed with Theorem~\ref{th:main}, where we list the isomorphism classes of projective surfaces with at most du Val singularities that admit an additive action with a finite number of orbits.
In particular, such a surface admits at most three 1-dimensional orbits and one fixed point, which is the only singularity if $X$ is not $\PP^2$ or $\FF_r$.
Moreover, any projective $\Ga^2$-surface with a finite number of orbits contains at most one singular point, see Remark~\ref{rm:unique-singularity}.
Finally, we describe surfaces from Theorem~\ref{th:main} that admit an additive action with infinite number of orbits, see Corollary~\ref{cr:inf-orbits}.

The author is grateful to Ivan Arzhantsev,  Anton Shafarevich, Mikhail Zaidenberg, and Yulia Zaitseva for helpful discussions and support.

\section{Equivariant models}\label{sec:prelim-models}
In this section we recall some general facts about equivariant completions of $\Ga^2$.
By a $\Ga^2$-action on a surface $X$ we mean an effective action $\Ga^2\curvearrowright X$ with an open orbit.
By a $\Ga^2$-surface we mean a projective surface with a choice of a $\Ga^2$-action on it.

We recall that a surface is called du Val if its singularities are rational double points. The exceptional divisor of its desingularization consists of $(-2)$-curves.

\begin{lemma}
   An irreducible curve $C$ on a smooth $\Ga^2$-surface with a negative self-in\-ter\-sec\-tion number is $\Ga^2$-stable.
\end{lemma}
\begin{proof}
   This follows from \cite[Proposition 2.3]{HaTs99}.
\end{proof}

\begin{lemma}\label{lm:lift-desing}
   Let $X$ be a normal projective surface,
    and $\eta\colon \widetilde{X}\to X$ be its minimal desingularization.
    Then any $\Ga^2$-action on $X$ is lifted via $\eta$ to $\widetilde X$ and vice versa.
\end{lemma}
\begin{proof}
   It is enough to repeat the proof of \cite[Lemma~4]{DeLo10}.
\end{proof}

\begin{proposition}[{\cite[Propositions~5.1--2]{HaTs99}}]\label{pr:HT-minimal}
   Any smooth $\Ga^2$-surface admits an equivariant morphism to the projective plane $\PP^2$, the direct product $\PP^1\times\PP^1$, or a Hirzebruch surface $\FF_r$, where $r\ge2$.
\end{proposition}

\begin{notation}\label{nt:Fr}
We denote by $\FF_0$ the direct product $\PP^1\times\PP^1$ and by $\FF_1$ the blowup of $\PP^2$ in one point.
For a $\GG_a^2$-action on $\FF_r$, $r\ge0$, 
the complement of the open orbit consists of a zero section $s$ with self-intersection number $(-r)$, and a distinguished fiber $f$ with self-intersection number $0$, see \cite{HaTs99}.
\end{notation}

\begin{corollary}\label{cr:Fr-model}
Any smooth $\Ga^2$-surface distinct from $\PP^2$ admits an equivariant morphism to $\FF_r$ for some $r\ge0$.
\end{corollary}
\begin{proof}
   It is enough to check that an equivariant morphism to $\PP^2$ is factored through one to $\FF_1$. 
   Then the assertion follows from Proposition~\ref{pr:HT-minimal}.
\end{proof}

In Sections~\ref{sec:local-coords} and~\ref{sec:blowup-models} we will work in the bubble space of equivariant blowups of $\FF_r$,
e.g., see \cite[Section~7]{DolgCAG} or \cite[Section~3]{Pe25}.

\begin{definition}
   Consider the category $\mathcal{B}_r$ of equivariant iterative blowups $\pi_X\colon X\to\FF_r$ of smooth $\Ga^2$-surfaces with morphisms over $\FF_r$ endowed with a $\Ga^2$-action.
   The set of morphisms between two surfaces $X,X'$ consists of at most one element, namely $\pi_{X'}^{-1}\circ \pi_X$.
   Denote
$$ 
\FF_r^{bb} = \sqcup_{X\in\mathcal{B}_r} X /\{p_1\sim p_2\},
$$
where $p_1\sim p_2$ if there exist neighbourhoods
$p_i\in U_i\subset X_i$ such that
$\pi_{X_2}^{-1}\circ\pi_{X_1}:U_1\cong U_2$, see \cite[Definition~7.3.1]{DolgCAG}. 
We say that a point $q\in\FF_r^{bb}$ is \emph{infinitesimally near} (or shortly \emph{near})  $p\in\FF_r^{bb}$ if $q\not\sim p$ and $\sigma(q)=p$ for some morphism $\sigma$ over $\FF_r$. 
We also say that $p\in \FF_r^{bb}$ is at the \emph{boundary} if $p$ does not belong to the open orbit.

The \emph{blowup} of points $p_1,\ldots,p_k\in\FF_r^{bb}$ is the smallest surface $X \in \mathcal{B}_r$ that does not contain $p_1,\ldots,p_k$.  
That is, for any surface $X'$ not containing $p_1,\ldots,p_k$ we have the morphism $X'\to X$ over $\FF_r$.
\end{definition}



\section{Additive actions in local coordinates}\label{sec:local-coords}
In this section we introduce local coordinates at infinitesimally near points at the boundary of the open orbit.

\begin{notation}\label{nt:toric-charts}
      Let us consider the toric charts of the $\FF_r$ as in \cite[Example~3.1.16]{CoLiSh11}.
      They are isomorphic to the affine plane.
      Consider the torus-fixed points $p_2:=s\cap f$, $p_1\in f\setminus\{p_2\}$, and $p_0\in \FF_r\setminus(s\cup f)$, see Notation~\ref{nt:Fr}.

      We introduce the following local coordinates at the toric charts with indicated points, which are compatible with the toric structure.
      \begin{itemize}
         \item The local coordinates $(x,y)$ at $p_0$, the $x$-axis (resp. $y$-axis) intersecting the zero section $s$ (resp. $f$) at infinity. 
         This toric chart is the open orbit of both $\Ga^2$-actions on $\FF_r$ considered below.
         \item The local coordinates $(u,v):=\left(\frac{x}{y^r},\frac{1}{y}\right)$ at $p_1$.
         The $u$-axis is the fiber $f$ and the $v$-axis lies in the open orbit.
         \item The local coordinates $(u,v):=\left(\frac{y^r}{x},\frac{1}{y}\right)$ at $p_2$.
         The $u$-axis is the fiber $f$ and the $v$-axis is the section $s$.
      \end{itemize}
\end{notation}


\begin{notation}\label{nt:coord-seq}
   We identify each point $p\in\FF_r^{bb}$, whose image in $\FF_r$ belongs to $f$, by a \emph{coordinate sequence} $[a_1,\ldots,a_k]$, where $a_i\in\K\cup\{\infty\}$. 
   We also endow each point with local coordinates.
   \begin{enumerate}[(i)]
         \item We let $[q]$, where $q\in\K$, denote the point $(q,0)\in f\setminus \{p_2\}$ in the local coordinates $(u,v)$ of $p_1$.
         In particular, $p_1$ is denoted by $[0]$.
         We endow $[q]$ with local coordinates $(u-q,v)$.
         \item We let $[\infty,\infty]$ denote the point $p_2$, the local coordinates being introduced above.
         \item Let $E$ be the exceptional curve of the blowup of a point $p$ with a coordinate sequence $[a_1,\ldots,a_k]$ and local coordinates $(u,v)$.
            Then we introduce a rational coordinate $u/v$ on $E$, so that its intersection with the preimage of the $v$-axis has coordinate $0$, and the one with the $u$-axis has coordinate $\infty$.
            We identify a point on $E$ with coordinate $q$ with a coordinate sequence $[a_1,\ldots,a_k,q]$.
            We endow the points on $E$ with local coordinates as follows.
            \begin{itemize}
               \item If $q=0$, then the local coordinates are $(u',v'):=(u/v,v)$, with $u'$-axis $\{v'=0\}$ being the exceptional curve $E$ and the $v'$-axis $\{u'=0\}$ being the preimage of the $v$-axis.
               \item If $q\neq0,\infty$, then the local coordinates are $(u/v-q,v)$, which differ from the case of $q=0$ by a shift.
               \item If $q=\infty$, then the local coordinates are $(u',v'):=(u,v/u)$, with $u'$-axis being the preimage of the $u$-axis and the $v'$-axis being $E$.
            \end{itemize}
         \item We shorten the coordinate sequence $[\ldots,\underbrace{a,\ldots,a}_{k},\ldots]$ to $[\ldots,a_{(k)},\ldots]$.
\end{enumerate}
\end{notation}

\begin{remark}
       A point $[\ldots,\infty]$ is always a node of the boundary divisor, 
      and the convention $[\infty,\infty]$ for $p_2$ comes from the blowup of the singular point of a weighted projective plane.
      Thus, coordinate sequences $[\infty,q]$ with $q\neq\infty$ are reserved to points of $s$ but never used.
\end{remark}

\begin{notation}\label{nt:curve-coord-seq}
 Consider an infinitesimally near point $p=[a_1,\ldots,a_k]$ at $\FF_r$. 
   Then we denote the exceptional curve $E$ of the blowup of $p$ by the same coordinate sequence. By abuse of notation, we denote $f$ by $[\,]$ and $s$ by $[\infty]$.
   
   Moreover, we introduce the local coordinates at the curve corresponding to the sequence $[a_1,\ldots,a_k]$ as ones of the point $[a_1,\ldots,a_k,0]$.
   In particular, $E$ is always the $u$-axis $\{v=0\}$ in its local coordinates.
\end{notation}


   \begin{definition}
         We say that the point $p\in\FF_r^{bb}$ is \emph{fixed} if the induced $\Ga^2$-action on $\FF_r^{bb}$ fixes $p$.
   We say that a curve at the boundary of $X\in\mathcal{B}_r$ is \emph{fixed} if it is pointwise fixed.
   \end{definition}

   \begin{remark}
      \begin{enumerate}[(i)]
         \item If the point $p=[a_1,\ldots,a_k]$ is fixed, then the vector fields of the considered $\Ga^2$-action are regular in the local coordinates of the exceptional curve at~$p$.
         \item 
   Assume that the point $p=[a_1,\ldots,a_k]$ is fixed and $a_{k+2},\ldots,a_n\in\{0,\infty\}$ for some $k,n\in\NN$.
   Then the sequence $$[a_1,\ldots,a_k,\infty,a_{k+2},\ldots,a_n]$$ 
   corresponds to a fixed point, since all blowups near $p$ occur at nodes of the boundary divisor. 
   Moreover, the exceptional curve at $p$ is also fixed.
      \end{enumerate}
\end{remark}

By \cite[Proposition~3.2]{HaTs99} and \cite[Proposition~5.5]{HaTs99}, there are two non-isomorphic actions 
on $\FF_r$ for $r\ge1$, one of them acting non-trivially on the distinguished fiber $f$, and there only one action on $\PP^1\times\PP^1$.
We recall them in  Remark~\ref{rm:xy} below.

\begin{definition}
We recall that two actions $\sigma_1,\sigma_2$ of an algebraic group $G$ on a variety $X$ are \emph{isomorphic} if 
there exist automorphisms $\phi$ of $X$ and $\psi$ of $G$ as an algebraic group such that the following diagram is commutative, e.g., see. \cite[Definition~1.9]{ArZa22}.
\begin{center}
\begin{tikzcd}
G \times X \arrow[r, "\sigma_1"] \arrow[d,"{(\psi,\phi)}"] & X \arrow[d, "\phi"] \\
G \times X \arrow[r, "\sigma_2"] & X
\end{tikzcd}   
\end{center}
\end{definition}

   \begin{remark}      \label{rm:xy}
   The subgroup $\Aut_f(\FF_r)$ of regular automorphisms of $\FF_r$ that preserve $f$ is given in the local coordinates $(x,y)$ at $p_0$ as follows.
\begin{equation}\label{eq:aut-f}
   \Aut_f(\FF_r)=\{(x,y)\mapsto(b_1x+a_0+a_1y+\ldots+a_ry^r,b_2y+c)\mid b_1,b_2\in\K^\times, a_i,c\in\K\}.
\end{equation}
   The two non-isomorphic $\Ga^2$-actions on $\FF_r$ are given in the coordinates $(x,y)$ by
   \begin{align}
      \phi_r\colon(a,b)\circ(x,y)=&(x+a,y+b),\label{eq:Ga2-standard}\\
      \psi_r\colon(a,b)\circ(x,y)=&\left(x+a +(y+b)^{r+1}-y^{r+1},y+b\right).
   \end{align}
   Both actions admit a 1-dimensional orbit $s\setminus\{p\}$.
   The action $\phi_r$ fixes $f$ pointwise for $r>0$ and $\psi_r$ acts on $f$ with a 1-dimensional orbit, cf. \cite[Proposition~5.5]{HaTs99}.
   If we blow up $(\FF_r,\psi_r)$ at the fixed point and contract the strict transform of $f$, then we obtain $(\FF_{r+1},\phi_{r+1})$ up to the coordinate change $(x,y)\mapsto(x+y^{r+1},y)$.
   So, $\psi_{r}$ is obtained from $\phi_{r+1}$ by equivariant blowup and blowdown for any $r\ge0$.
   In addition, in the case of $\FF_0$ the actions $\psi_0$ and $\phi_0$ are isomorphic.
\end{remark}

\begin{remark}\label{rm:actions}
   We can compute explicitly a $\Ga^2$-action in any local coordinates provided that the blown up points are fixed, using Notation~\ref{nt:coord-seq}.
   
   For example, in the local coordinates at $[\infty,\infty]$
    the $\Ga^2$-actions are given by
   \begin{alignat}{3}
      \phi_r\colon(a,b)\circ(u,v) &=\bigg(\frac{u(1+bv)^r}{1+auv^r},& \frac{v}{1+bv}\bigg),\\
      \psi_r\colon(a,b)\circ(u,v) &=\bigg(\frac{u(1+bv)^r}{1+auv^r+\frac{u}{v}\left((1+bv)^{r+1}-1\right)},& \frac{v}{1+bv}\bigg).
   \end{alignat}
  
   The corresponding vector fields are
   \begin{alignat}{5}
      \partial_a \phi_r =& (-u^2v^r,&0),\quad
      &\partial_b \phi_r =& (ruv,-v^2),\\
      \partial_a \psi_r =& (-u^2v^r,&0),\quad
      &\partial_b \psi_r =& (-ru^2+ruv,-v^2).
   \end{alignat}
      One could notice 
      that the axes of local coordinates at $p=[\infty,\infty]$ are indeed stable,
   and $\phi_r$ (pointwise) fixes the $u$-axis, i.e., the fiber $f$.
   \end{remark}

\section{Smooth equivariant blowup models}\label{sec:blowup-models}
Here we describe equivariant blowups of $\FF_r$, whose fixed components have self-intersection number $-2$.
This will allow us to describe $\Ga^2$-surfaces with ADE-singularities and a finite number of orbits in Theorem~\ref{th:main}.

\begin{lemma}\label{lm:contract-to-Fr-phi}
   Let $X$ be a smooth $\Ga^2$-surface that is not isomorphic to $(\FF_r,\psi_r)$ for $r\ge0$ or to $\PP^2$.
   Then $X$ is obtained from $(\FF_r,\phi_r)$ for some $r\ge1$
   by a sequence of blowups of fixed points.
\end{lemma}
\begin{proof}
   By Corollary~\ref{cr:Fr-model}, there is a sequence of equivariant blowups 
   $$X\stackrel{\sigma_1}{\to} X_1\stackrel{\sigma_2}{\to}\cdots\stackrel{\sigma_d}{\to} X_d=\FF_r$$ 
   for some $d,r\ge0$.
   Assume that the resulting $\Ga^2$-action on $\FF_r$ is isomorphic to $\psi_r$, 
   otherwise we have $r\ge1$ and the statement holds.
   Then $\sigma_d$ is the blowup of the only fixed point $[\infty,\infty]$.
   If so, $X_{d-1}$ admits a contraction to $\FF_{r+1}$, and the resulting action on $\FF_{r+1}$ fixes $f$, hence is isomorphic to $\phi_{r+1}$.
\end{proof}

\begin{corollary}
   Let $X$ be a smooth $\Ga^2$-surface that is not isomorphic to $\PP^1\times\PP^1$.
   Then there is an equivariant morphism of $X$ to the weighted projective plane $\PP(1,1,r)$ for some $r\ge1$.
\end{corollary}
\begin{proof}
   The assertion follows from Lemma~\ref{lm:contract-to-Fr-phi}, since contracting the zero section of a Hirzebruch surface yields a weighted projective plane.
\end{proof}

The computations in the following lemma are assisted by SageMath \cite{Sage}.

\begin{lemma}\label{lm:fixed-curves}
   Consider the $\Ga^2$-surface $(\FF_r,\phi_r)$, where $r\ge1$.
   Then we have the following:  
   \begin{enumerate}[(i)]
      \item The exceptional component at $[q]$, where $q\neq\infty$, is fixed if and only if $q=0$ and $r>1$.
      \item The exceptional component $E$ at $p=[0_{(k)},q]$, where $q\neq\infty$ and $k<r-1$, is fixed if and only if $q=0$.
      \item The exceptional component $E$ at $p=[0_{(r-1)},q]$, where $q\neq\infty$,  is never fixed.   
      \item The exceptional component at $[q,\infty,w]$, where $q\neq\infty$ and $w\neq0,\infty$, is fixed if and only if $q=0$ and $r>1$.
      \item Assuming that $r>1$, the exceptional component at $[0,\infty,w,q]$, where $q\neq\infty$ and $w\neq0,\infty$, is never fixed.
      \item The exceptional component at $[\infty,\infty,0_{(k)},w]$, where $w\neq0,\infty$ and $k\ge0$, is never fixed.
      \item The exceptional component at $[\infty,\infty,\infty,w]$, where $w\neq0,\infty$, is always fixed.
      \item The exceptional component at $[\infty,\infty,\infty,w,q]$, where $w\neq0,\infty$ and $q\neq\infty$, is never fixed.
   \end{enumerate}
\end{lemma}
\begin{proof}
   (ii--iii) Let us denote $t:=r-1-k$ for some non-negative $k\le r-1$. 
   By induction by $k$ we deduce that the action vector fields in the local coordinates of the point $[0_{(k)},q,0]$ are as follows,
   \[
      \partial_a \phi_r =  (v^t, 0),\quad
      \partial_b \phi_r =(-tuv-(t+1)q, -v^2).
   \]
   We see that $\partial_a \phi_r$ is zero on $E$ if and only if $t>0$, and $\partial_b \phi_r$ is zero on $E$ if and only if $q=0$.
   Then the claims follow, and (i) is their special case $k=0$.

   (iv) Let $W=uv+w$ and $Q=Wv+q$. Then the action vector fields in the local coordinates of the point $[q,\infty,w,0]$ are as follows,
   \begin{alignat*}{4}
      \partial_a \phi_r =&  (W^{r-1}v^{2r-2}(uv+2W),& -W^{r-1}v^{2r}),&\\
      \partial_b \phi_r =&(Wuv^2+W^2v-rQ(uv+2W),& -Wv^3+rQv^2).&
   \end{alignat*}
   Then $Q=q$ and $W=w$ at $\{v=0\}$, hence $\partial_b \phi_r|_{v=0}$ equals $(-2rqw,0)$ and $\partial_a \phi_r|_{v=0}$ equals $(2w,0)$ if $r=1$ and zero otherwise.
   The assertion follows.  

   (v) Analogously, direct computations show that 
   in the local coordinates of the point $[0,\infty,w,q,0]$ we have
    $\partial_a \phi_r|_{v=0}=0$ and $\partial_b \phi_r|_{v=0}=((1-2r)w^2,0)$, which is never zero.


   (vi) The action vector fields in the local coordinates of the point $[\infty,\infty,0_{(k)},w]$ are as follows,
   \[
      \partial_a \phi_r =  (-(uv+w)^2v^{r+k}, 0),\quad
      \partial_b \phi_r =((r+k+2)uv+(r+k+1)w,-v^2).
   \]
   We see that $\partial_b \phi_r|_{v=0}$ is never zero.

   (vii--viii)  Let $Q=uv+q$ and $W=Qv+w$. Then the action vector fields in the local coordinates of the point $[\infty,\infty,\infty,w,q,0]$ are as follows,
   \begin{alignat*}{4}
      \partial_a \phi_r =&( -W^{r+1}v^{2r-1}(-uv^2-Qv+2W) ,& W^{r+1}v^{2r+2} ),&\\
      \partial_b \phi_r =&( (r+1)Wv(uv+Q)+(2r+1)W^2 ,& -(r+1)Wv^3 ).&
   \end{alignat*}
   Since they are regular, the point $[\infty,\infty,\infty,w]$ is fixed.
   Finally,  $\partial_a \phi_r|_{v=0}$ is zero and $\partial_b \phi_r|_{v=0}$ equals $((2r+1)w^2,0)$, which is never zero.
\end{proof}

\begin{proposition}\label{pr:blowups}
   Let $X$ be a smooth $\Ga^2$-surface that is obtained by equivariant blowups of fixed points from $(\FF_r,\phi_r)$ for some $r\ge1$.
   Assume that every fixed irreducible curve at the boundary of $X$ is of self-intersection number $-2$.
   Then one of the following is true:  
   \begin{enumerate}[(i)]
      \item $X$ is the blowup of a pair of points at the boundary $p_1,p_2$, whose images in $\FF_r$ do not coincide, and which are of the following form:
      \begin{itemize}
         \item either $[0_{(k)},q]$, where $q\neq0,\infty$ and $k\le r-1$;
         \item or $[0_{(r)}]$;
         \item or $[\infty,\infty,0_{(k)},w]$, where $w\neq0,\infty$ and $k\ge0$.
      \end{itemize}
      \item $X$ is the blowup of one of the following points:
      \begin{enumerate}[(a)]
         \item $[0,\infty,w,q]$, where $w\neq0,\infty$ and $q\neq\infty$, assuming that $r>1$;
         \item $[\infty,\infty,\infty,w,q]$, where $w\neq0,\infty$ and $q\neq\infty$;
         \item  $[q,\infty,w]$, where $q,w\neq0,\infty$;
         \item $[0,\infty,w]$, where $w\neq0,\infty$, assuming that $r=1$.
      \end{enumerate}
   \end{enumerate}
   The corresponding  dual graphs are depicted in Fig.~\ref{fig:blowups}.
\end{proposition}

\begin{figure}
   \centering
\subcaptionbox{Case (i) without blowing up $[\infty,\infty]$}
{\begin{tikzpicture}[scale=1.5, 
    every node/.style={circle, draw=black, fill=gray, inner sep=2pt, font=\tiny},
    every label/.style={rectangle, draw=white, fill=white, inner sep=2pt, font=\tiny},
     thick]

\node[rectangle, label=above:{$[0_{(k)},q_1]$}, label=below:{$-1$}] (2) at (0,0) {};
\node[circle, label=above:{$[0_{(k)}]$}, label=below:{$-2$}] (3) at (1,0) {};
\node[circle, label=above:{$[0]$}, label=below:{$-2$}] (4) at (2,0) {};
\node[label=above:{$[\,]$}, label=below right:{$-2$}] (5) at (3,0) {};
\node[rectangle,label=above:{$[\infty]$}, label=below:{$-r$}] (r) at (4,0) {}; 
\node[rectangle,label=left:{$[q_2]$}, label=right:{$-1$}] (6) at (3,-1) {};

\draw (2) -- (3) edge[dotted] (4);
\draw (4) -- (5) -- (r);
\draw (5) -- (6) ;

\end{tikzpicture}}
\subcaptionbox{Case (i) with blowing up $[\infty,\infty]$}
{\begin{tikzpicture}[scale=1.5, 
    every node/.style={circle, draw=black, fill=gray, inner sep=2pt, font=\tiny},
    every label/.style={rectangle, draw=white, fill=white, inner sep=2pt, font=\tiny},
     thick]

\node[rectangle, label=above:{$[0_{(k_1)},q]$}, label=below:{$-1$}] (2) at (0,0) {};
\node[circle, label=above:{$[0_{(k_1)}]$}, label=below:{$-2$}] (3) at (1,0) {};
\node[circle, label=above:{$[0]$}, label=below:{$-2$}] (4) at (2,0) {};
\node[label=above:{$[\,]$}, label=below:{$-2$}] (5) at (3,0) {};
\node[label=above:{$[\infty_{(2)}]$}, label=below:{$-2$}] (6) at (4,0) {};
\node[label=above:{$[\infty_{(2)},0]$}, label=below:{$-2$}] (7) at (5,0) {};
\node[label=above:{$[\infty_{(2)},0_{(k_2)}]$}, label=below left:{$-2$}] (8) at (6,0) {};
\node[rectangle,label=above:{$[\infty]$}, label=below:{$-r-k_2-1$}] (r) at (7,0) {}; 
\node[rectangle,label=left:{$[\infty_{(2)},0_{(k_2)},w]$}, label=right:{$-1$}] (0) at (6,-1) {};

\draw (2) -- (3) edge[dotted] (4);
\draw (4) -- (5) -- (6) -- (7);
\draw (7) edge[dotted] (8);
\draw (0) -- (8) -- (r);

\end{tikzpicture}}
\subcaptionbox{Case (ii.a), where $r>1$}{\begin{tikzpicture}[scale=1.5, 
    every node/.style={circle, draw=black, fill=gray, inner sep=2pt, font=\tiny},
    every label/.style={rectangle, draw=white, fill=white, inner sep=2pt, font=\tiny},
     thick]

\node[rectangle, label=above:{$[0,\infty,w,q]$}, label=below:{$-1$}] (0) at (0,0) {};
\node[circle, label=above:{$[0,\infty,w]$}, label=below:{$-2$}] (1) at (1,0) {};
\node[circle, label=above:{$[0,\infty]$}, label=below right:{$-2$}] (2) at (2,0) {};
\node[label=above:{$[\,]$}, label=below:{$-2$}] (3) at (3,0) {};
\node[rectangle,label=above:{$[\infty]$}, label=below:{$-r$}] (r) at (4,0) {}; 
\node[label=left:{$[0]$}, label=right:{$-2$}] (d) at (2,-1) {};

\draw (0) -- (1) -- (2) -- (3) -- (r) ;
\draw (d) -- (2);

\end{tikzpicture}}
\subcaptionbox{Case (ii.b)}{\begin{tikzpicture}[scale=1.5, 
    every node/.style={circle, draw=black, fill=gray, inner sep=2pt, font=\tiny},
    every label/.style={rectangle, draw=white, fill=white, inner sep=2pt, font=\tiny},
     thick]

\node[rectangle, label=above:{$[\infty_{(3)},w,q]$}, label=below:{$-1$}] (0) at (0,0) {};
\node[circle, label=above:{$[\infty_{(3)},w]$}, label=below:{$-2$}] (1) at (1,0) {};
\node[circle, label=above:{$[\infty_{(3)}]$}, label=below right:{$-2$}] (2) at (2,0) {};
\node[label=above:{$[\infty_{(2)}]$}, label=below:{$-2$}] (3) at (3,0) {};
\node[rectangle,label=above:{$[\infty]$}, label=below:{$-r-1$}] (r) at (4,0) {}; 
\node[label=left:{$[\,]$}, label=right:{$-2$}] (d) at (2,-1) {};

\draw (0) -- (1) -- (2) -- (3) -- (r) ;
\draw (d) -- (2);

\end{tikzpicture}}
\subcaptionbox{Case (ii.c)}{\begin{tikzpicture}[scale=1.5, 
    every node/.style={circle, draw=black, fill=gray, inner sep=2pt, font=\tiny},
    every label/.style={rectangle, draw=white, fill=white, inner sep=2pt, font=\tiny},
     thick]

\node[rectangle, label=above:{$[q,\infty,w]$}, label=below:{$-1$}] (1) at (1,0) {};
\node[circle, label=above:{$[q,\infty]$}, label=below right:{$-2$}] (2) at (2,0) {};
\node[label=above:{$[\,]$}, label=below:{$-2$}] (3) at (3,0) {};
\node[rectangle,label=above:{$[\infty]$}, label=below:{$-r$}] (r) at (4,0) {}; 
\node[label=left:{$[q]$}, label=right:{$-2$}] (d) at (2,-1) {};

\draw (1) -- (2) -- (3) -- (r) ;
\draw (d) -- (2);

\end{tikzpicture}}
\subcaptionbox{Case (ii.d)}{\begin{tikzpicture}[scale=1.5, 
    every node/.style={circle, draw=black, fill=gray, inner sep=2pt, font=\tiny},
    every label/.style={rectangle, draw=white, fill=white, inner sep=2pt, font=\tiny},
     thick]

\node[rectangle, label=above:{$[0,\infty,w]$}, label=below:{$-1$}] (1) at (1,0) {};
\node[circle, label=above:{$[0,\infty]$}, label=below right:{$-2$}] (2) at (2,0) {};
\node[label=above:{$[\,]$}, label=below:{$-2$}] (3) at (3,0) {};
\node[rectangle,label=above:{$[\infty]$}, label=below:{$-1$}] (r) at (4,0) {}; 
\node[rectangle,label=left:{$[0]$}, label=right:{$-2$}] (d) at (2,-1) {};

\draw (1) -- (2) -- (3) -- (r) ;
\draw (d) -- (2);

\end{tikzpicture}}
   \caption{The incidence graph of the components of the open orbit complement for surfaces in Proposition~\ref{pr:blowups}. 
   Fixed components are depicted by circles and non-fixed ones by squares.
   Each component is labeled by its self-intersection number and by the corresponding coordinate sequence, see Notation~\ref{nt:curve-coord-seq}. 
   The strict transforms of $f$ and $s$ are labeled by $[\,]$ and $[\infty]$ respectively.  
   }\label{fig:blowups}
\end{figure}
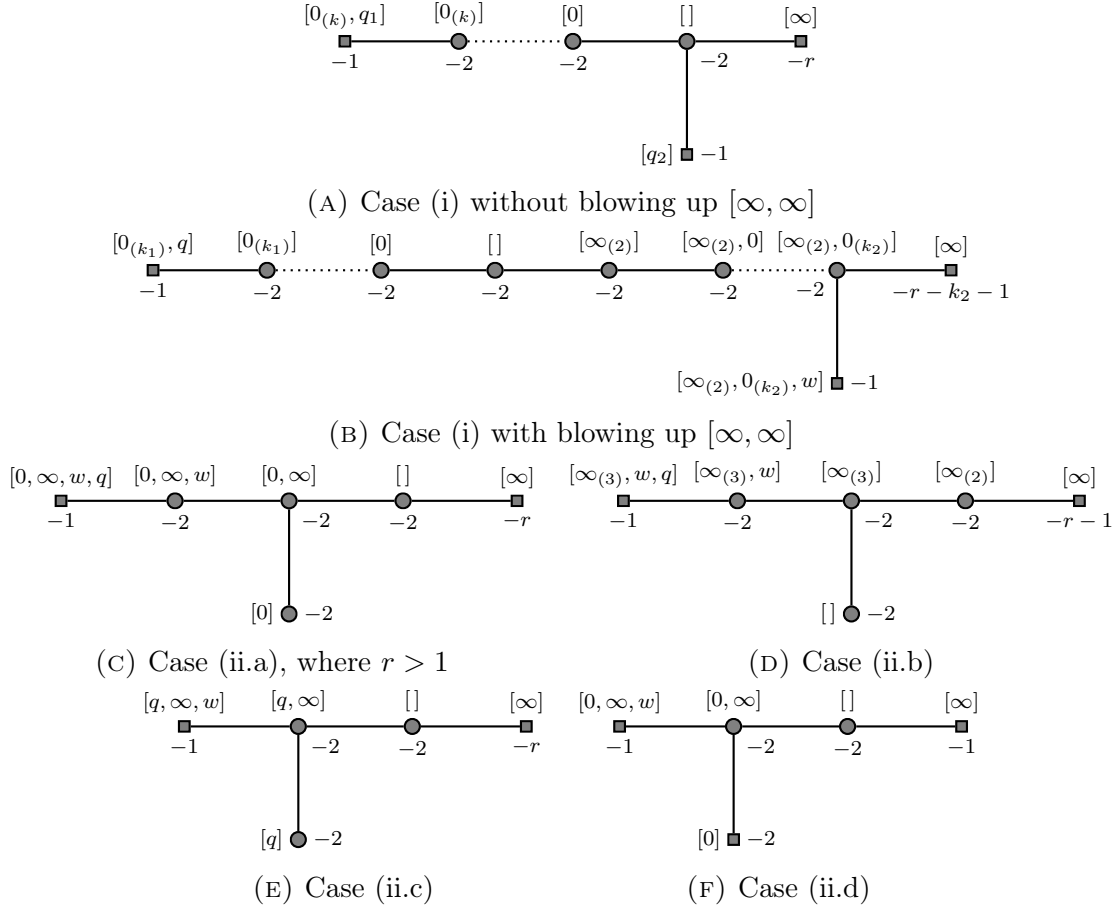

\begin{proof}
The fiber $f$ is fixed and of self-intersection number $0$.
Then its strict transform in $X$ has self-intersection number $-2$.

Assume first that $\sigma_X^{-1}\colon \FF_r\dashrightarrow X$ blows up two distinct points of $f$. 
Denote one of them by $p$
and consider the subsequence of blowups $\sigma'\colon\FF_r\dashrightarrow X'$ that occur at $p$ or near it.
Then the following cases are possible.
\begin{itemize}
   \item If the exceptional curve at $p=[q]$ is not fixed, then by Lemma~\ref{lm:fixed-curves}(iv), we have either $q\neq0,\infty$, or $q=0$ and $r=1$. 
   Notice that the blowup of a non-intersection point on a non-fixed component is not equivariant.
   Then $\sigma'$ consists of the single blowup of $[q]$.
   \item If $p=[0]$ and $r>1$, then  we are allowed to iteratively blow up  the non-intersection points of the resulting exceptional curves, starting with one at $p$, until we obtain a non-fixed curve.
   Indeed, we cannot blow up a point on a fixed $(-2)$-curve.
    By Lemma~\ref{lm:fixed-curves}(ii--iii), $\sigma'$ is the blowup of the infinitesimally near point $[0_{(k)},q]$, 
    where either $q\neq0,\infty$ and $k\le r-1$ or  $q=0$ and $k=r-1$.
   \item If $p=[\infty,\infty]$ and the exceptional curve at $p$ is $E$, then $E$ itself is fixed, but any blowup of a non-intersection point is non-fixed by Lemma~\ref{lm:fixed-curves}(vi). 
   If we continue blowing up intersection points, then we are allowed to blow up only the one lying on the strict transform of $s$, thus blowing up the point $[\infty,\infty,0_{(k)}]$. 
   By Lemma~\ref{lm:fixed-curves}(vi), the blowup at $[\infty,\infty,0_{(k)},w]$, where $w\neq0,\infty$, provides a non-fixed component, hence we stop.
\end{itemize}
This finishes case (i). 
Now assume that all blowups in the sequence of blowups of points defined by $\sigma_X^{-1}$ occur near the same point of $f$.
We may also assume that the second blowup in $\sigma_X^{-1}$ occurs at a point $p$ on the strict transform of $f$.
Let $E_1$ be the exceptional curve of the first blowup and $E_2$ be the one of the second blowup, which is at $p$.
In particular, $p$ is the intersection point of $E_1$ and the strict transform of $f$.
 
   If $E_1$ is fixed, then it is a $(-2)$-curve after the second blowup, and all subsequent blowups occur at non-intersection points of the emerging exceptional curves, starting with $E_2$ and finishing with a non-fixed curve.
   By Lemma~\ref{lm:fixed-curves}(i), $p$ equals either $[0,\infty]$ under the condition $r>1$, or $[\infty,\infty,\infty]$.
   In the former case, $\sigma_X^{-1}$ is the blowup of the point $[0,\infty,w,q]$, where $w\neq0,\infty$ and $q\neq\infty$, see Lemma~\ref{lm:fixed-curves}(iv--v).
   In the latter case,  $\sigma_X^{-1}$ is the blowup of the point $[\infty,\infty,\infty,w,q]$, where $w\neq0,\infty$ and $q\neq\infty$, see Lemma~\ref{lm:fixed-curves}(vii--viii). 
   
   We are left with the case when $E_1$ is not fixed.
   Then either $p=[q,\infty]$ for some $q\neq0,\infty$ or $p=[0,\infty]$ and $r=1$.
   In both cases the exceptional component at a non-intersection point on $E_2$ is non-fixed, see Lemma~\ref{lm:fixed-curves}(iv).
   Thus, in the former case $\sigma_X^{-1}$ is the blowup of $[q,\infty,w]$ for some $w\neq0,\infty$, and
   in the latter case $\sigma_X^{-1}$ is the blowup of $[0,\infty,w]$ for some $w\neq0,\infty$.
\end{proof}

Let us show that cases $(B,D,F)$ of Fig.\ref{fig:blowups} are redundant.

\begin{notation}
   By $Bl(\FF_r,p_1,\ldots,p_k)$ we denote a $\Ga^2$-surface obtained by the blowup of $(\FF_r,\phi_r)$ at points $p_1,\ldots,p_k\in\FF_r^{bb}$.
\end{notation}

\begin{corollary}\label{cr:cases-ACE}
   A $\Ga^2$-surface $X$ as in Proposition~\ref{pr:blowups} up to an equivariant isomorphism belongs to one of the following cases, enumerated after ones in Fig.\ref{fig:blowups}: 
   \begin{itemize}
      \item[(A)] $Bl(\FF_r,[0_{(k)},q_1],[q_2])$,
      where $k\le r-1$, $q_1,q_2\neq\infty$, $q_1\neq0$ if $k<r-1$, and 
      $q_2\neq\begin{cases}
         q_1,\text{ if } k=0,\\
         0,\text{ if } k>0;
      \end{cases}$ 
      \item[(C)] $Bl(\FF_r,,[0,\infty,w,q])$, where $r>1$, $w\neq0,\infty$, and $q\neq\infty$;
      \item[(E)] $Bl(\FF_r,[q,\infty,w])$, where $q,w\neq0,\infty$.
   \end{itemize}
   In addition, the surfaces corresponding to different cases are not isomorphic.
\end{corollary}
\begin{proof}
   In the case and notation of Fig.\ref{fig:blowups}(B), we may contract all components except $[\infty_{(2)},0_{(k_2)}]$ and $[\infty]$.
   Then we obtain $(\FF_{r+k_2+1},\phi_{r+k_2+1})$, since $[\infty_{(2)},0_{(k_2)}]$ is fixed, and the blowup $X\to \FF_{r+k_2+1}$ falls into the case of Fig.\ref{fig:blowups}(A).

   Analogously, in the case of Fig.\ref{fig:blowups}(D) we may contract all components except $[\infty_{(2)}]$ and $[\infty]$ and fall into the case of Fig.\ref{fig:blowups}(C).
   Finally, in the case of Fig.\ref{fig:blowups}(F) we may contract all components except $[0]$ and $[0,\infty]$ and fall into the case of Fig.\ref{fig:blowups}(A) for $r=2$.

   By checking weights of dual graphs and non-fixed components one may conclude that the remaining cases provide non-isomorphic surfaces.
\end{proof}


   
\section{Isomorphism classes}\label{sec:isom-classes}


We describe isomorphism classes of obtained surfaces by studying the action of $\Aut_f(\FF_r)$ \eqref{eq:aut-f} and the $\Ga^2$-action $\phi_r$ \eqref{eq:Ga2-standard} 
on $\FF_r^{bb}$ in the coordinates $(x,y)$ of the open orbit.

\begin{lemma}\label{lm:norm-isom}
Let $\pi_i\colon (X_i,\gamma_i)\to (X,\gamma)$, $i=1,2$, be equivariant iterative blowups of a $\Ga^2$-surface  $(X,\gamma)$.
 Denote by $P_i$ the corresponding sets of blown up points on $X$ (including infinitesimally near ones). 

 Then there exists an isomorphism (resp. $\Ga^2$-equivariant isomorphism) $\tau\colon X_1\to X_2$ that sends components contracted by $X_1$ into components contracted by $X_2$ 
if and only if there exists an automorphism $g\in \Aut(X)$  (resp. $g$ in the normalizer of $\gamma(\Ga^2,\cdot)\subset\Aut(X)$) that induces a bijection of $P_1$ and $P_2$.
\end{lemma}
\begin{proof}
   An existence of such $\tau$ (without the equivariance condition) induces the existence of such $g\in \Aut(X)$ and vice versa, as the following commutative diagram suggests.
   \begin{center}
   \begin{tikzcd}
      X_1 \arrow[r, "\tau"] \arrow[d, "\pi_1"] & X_2 \arrow[d, "\pi_2"] \\
      X \arrow[r, "g"] & X
   \end{tikzcd}      
   \end{center}
   Moreover, $\tau$ is $\Ga^2$-equivariant if and only if $g$ is $\gamma$-equivariant.
   We are left to check that $g$ is $\gamma$-equivariant if and only if $g$ is in the normalizer of the image of $\gamma$ in $\Aut(X)$.
  This is left to the reader. 
\end{proof}

\begin{lemma}\label{lm:norm-Fr}
The normalizer of $\phi_r$ in $\Aut_f(\FF_r)$ is as follows. 
\begin{equation}\label{eq:norm-Fr}
   N({\phi_r}) = \{(x,y)\mapsto (b_1x+a_0+a_1y,b_2y+c)\mid b_1,b_2\in\K^\times, a_0,a_1,c\in\K\}.
\end{equation}
\end{lemma}
\begin{proof}
   This is a straightforward check in formulas \eqref{eq:aut-f} and \eqref{eq:Ga2-standard}.
\end{proof}

\begin{lemma}\label{lm:norm-P2}
   Let $\pi\colon \FF_1\to\PP^2$ be the contraction of the zero fiber.
   Then the normalizer of the image $H$ of $\phi_1$ in $\Aut(\PP^2)$ equals 
   \[
   N_{\PP^2}(\phi_1)=\mathrm{GL}_2(\K)\ltimes H,
   \]
   where $\mathrm{GL}_2(\K)$ acts on $\PP^2$ by linear transformations of the open $H$-orbit centered at the image of the origin $p_0\in\FF_1$.
\end{lemma}
\begin{proof}
   The normalizer preserves the open orbit, hence is contained to the subgroup 
   \begin{equation}\label{eq:norm-P2}
      \{(x,y)\mapsto (a_{1,1}x+a_{1,2}y+b_1,a_{2,1}x+a_{2,2}y+b_2)\mid b_1,b_2\in\K,\left(\begin{smallmatrix}
      a_{1,1}&a_{1,2}\\
      a_{2,1}&a_{2,2}
      \end{smallmatrix}\right)\in\mathrm{GL}_2(\K)\}    
   \end{equation}
   of $\Aut(\PP^2)$. 
   On the other hand, this subgroup normalizes the image of $\phi_1$. 
   The assertion follows.
\end{proof}

\begin{remark}\label{rm:norm-action}
The normalizer $N(\phi_r)$ acts on points at the boundary in the following way in coordinates \eqref{eq:norm-Fr}:
\begin{enumerate}[(i)]
   \item $[q]\mapsto \left[ \frac{b_1}{b_2^r}q\right]$, where $r>1$;
   \item $[0_{(k)},q]\mapsto \left[0_{(k)},\frac{b_1}{b_2^{r-k}}q\right]$, where $k<r$;
   \item $[0_{(r-1)},q]\mapsto \left[0_{(r-1)},\frac{b_1q+a_1}{b_2}\right]$;
   \item $[0,\infty,w]\mapsto \left[0,\infty,\frac{b_1^2}{b_2^{2r-1}}w\right]$, where $w\neq0,\infty$ and $r>1$;
   \item the stabilizer of $[q]$ defined by $b_1=b_2^r$ if $r>1$ (resp. $a_1=b_2-qb_1$ if $r=1$) sends $[q,\infty,w]$ to $[q,\infty,b_2w]$ (resp. $[q,\infty,\frac{b_1^2w}{b_2}]$), where $q,w\neq0,\infty$;
   \item the stabilizer of $[0,\infty,w]$, where $w\neq0,\infty$,  defined by $b_1^2=b_2^{2r-1}$, sends $[0,\infty,w,q]$ to $[0,\infty,w,bq]$ if $r>2$ and to $[0,\infty,w,bq+\frac{2wa_{1}}{b^{2}}]$ if $r=2$, where  $q\neq\infty$ and $b=\frac{b_2^r}{b_1}$ satisfies $b_1=b^{2r-1},b_2=b^2$.
\end{enumerate}
\end{remark}

Let us apply Lemma~\ref{lm:norm-isom} to Corollary~\ref{cr:cases-ACE} in a case by case manner.

\begin{proposition}\label{pr:cases-eq-isom}
   A smooth $\Ga^2$-surface $X$ as in Corollary~\ref{cr:cases-ACE} is equivariantly isomorphic to exactly one (except for case (ii))  of the following surfaces:
   \begin{enumerate}[(i)]
      \item $Bl(\FF_1,[0],[1])$;
      \item $Bl(\FF_r,[1],[q])$, where $r\ge2$ and $q\neq\infty$ is defined uniquely if $q=0$ and up to taking inverse otherwise;
      \item $Bl(\FF_r,[0_{(k)},1],[1])$, where $r\ge2$ and $k\in\{1,\ldots,r-1\}$;
      \item $Bl(\FF_r,[0,\infty,1,0])$, where $r\ge2$;
      \item $Bl(\FF_r,[0,\infty,1,1])$, where $r\ge3$;
      \item $Bl(\FF_r,[1,\infty,1])$, where $r\ge1$.
   \end{enumerate}
\end{proposition}
\begin{proof}

   We go through the cases of Corollary~\ref{cr:cases-ACE}.

(i) In the case (A) with $k=0,r=1$, we have $X=Bl(\FF_1,[q_1],[q_2])$ for some distinct $q_1,q_2\neq\infty$.
By Lemma~\ref{lm:norm-P2}, the normalizer $N_{\PP^2}(\phi_1)$ acts transitively on triples of points at the image of $f$.
   By Lemma~\ref{lm:norm-isom}, we see after blowing down to $\PP^2$ that $X$ is equivariantly isomorphic to $Bl(\FF_1,[0],[1])$.

(ii) In the case (A) with $k=0,r>1$, we have $X=Bl(\FF_r,[q_1],[q_2])$ for some distinct $q_1,q_2\neq\infty$. 
   By Remark~\ref{rm:norm-action}(i), $X$ is equivariantly isomorphic to $Bl(\FF_r,[1],[q])$ if and only if $q$ equals $\frac{q_1}{q_2}$  or $\frac{q_2}{q_1}$.
   The statement follows.

(iii) In the case (A) with $k>0$, we have either $X=Bl(\FF_r,[0_{(k)},q_1],[q_2])$ or $X=Bl(\FF_r,[0_{(r)}],[q_2])$ for some $q_1,q_2\neq0,\infty$.
By Remark~\ref{rm:norm-action}(i--iii), $X$ is equivariantly isomorphic to $Bl(\FF_r,[0_{(k)},1],[1])$.

(iv,v)  In the case (C) we have $r>1$ and $X=Bl(\FF_r,[0,\infty,w,q])$ for some $w\neq0,\infty$ and $q\neq\infty$. 
By Remark~\ref{rm:norm-action}(iv), $X$ is equivariantly isomorphic to $X'=Bl(\FF_r,[0,\infty,1,q'])$ for some $q'\neq\infty$.
By Remark~\ref{rm:norm-action}(vi), $X'$ is equivariantly isomorphic to exactly one of the surfaces $Bl(\FF_r,[0,\infty,1,0])$ or $Bl(\FF_r,[0,\infty,1,1])$ if $r>2$ and to $Bl(\FF_r,[0,\infty,1,0])$ if $r=2$.

(vi) In the case (E) we have $X=Bl(\FF_r,[q,\infty,w])$ for some $q,w\neq0,\infty$.
By Remark~\ref{rm:norm-action}(i,iii), 
$X$ is equivariantly isomorphic to $X'=Bl(\FF_r,[1,\infty,w'])$ for some $w'\neq0,\infty$.
By Remark~\ref{rm:norm-action}(v), $X'$ is equivariantly isomorphic  to $Bl(\FF_r,[1,\infty,1])$.

Finally, it is clear from the dual graphs of the boundaries that surfaces from different cases are not equivariantly isomorphic.
\end{proof}

\begin{lemma}\label{lm:abs-isom-1-q}
   A surface $Bl(\FF_r,[1],[q])$, where $r>1$ and $q\neq1,\infty$, is isomorphic to $Bl(\FF_r,[0],[1])$.
\end{lemma}
\begin{proof}
   The automorphism group $\Aut_f(\FF_r)$ acts by sending a point $[q]$, $q\neq\infty$, to $[\frac{b_1q+a_r}{b_2^r}]$ in terms of \eqref{eq:aut-f}, hence acts transitively on pairs of points on $f\setminus\{f\cap s\}$. The assertion follows.
\end{proof}

\begin{lemma}\label{lm:abs-isom-011}
   Surfaces $Bl(\FF_r,[0,\infty,1,0])$ and $Bl(\FF_r,[0,\infty,1,1])$ are isomorphic, where $r>1$.
\end{lemma}
\begin{proof}
   The stabilizer of $[0,\infty,1]$ in $\Aut_f(\FF_r)$ is defined by equations $a_r=0$ and $b_1^2=b_2^(2r-1)$ in terms of \eqref{eq:aut-f}.
   Letting $b=\frac{b_2^r}{b_1}$, we see that the stabilizer sends $[0,\infty,1,q]$ into $[0,\infty,1,bq+\frac{2a_{r-1}}{b^{2(r-1)}}]$. The assertion follows.
\end{proof}

\begin{theorem}\label{th:main}
   The (abstract) isomorphism classes of projective surfaces $Y$ with ADE-sin\-gu\-la\-ri\-ti\-es that admit an additive action with a finite number of orbits are listed in Table~\ref{tb:isom-classes}.
   In particular, such a surface $Y$ admits at most three 1-dimensional orbits and the only fixed point $p$.
   If $Y$ is not isomorphic to $\PP^2$ or $\FF_r$, where $r\ge0$, then $p$ is the only singularity of $Y$.

   Moreover, each surface in Table~\ref{tb:isom-classes} admits (up to an equivariant automorphism) exactly one $\Ga^2$-action with a finite number of orbits, except for the following cases:
   \begin{itemize}
      \item Nos. 3 and 8 admit two non-isomorphic actions for $r\ge3$;
      \item Nos. 6 and 11 admit a 1-dimensional family of non-isomorphic actions.
   \end{itemize}   
\end{theorem}

   \begin{table}[ht]
   \begin{tblr}{|c|c|c|p{2cm}|c|c|c|}
      \hline
      no.&case & desingularization & parameters & action & 1-dim. orbits & singularity \\
      \hline\hline
      1&$\PP^2$ &$\PP^2$ & -- & 
      $\psi_1$ & $[\,]$ & smooth \\
      2&$\PP(1,1,2)$ &$\FF_2$ & -- & 
      $\psi_2$ 
       & $[\,]$ & $A_1$ \\
   3& (C) & $Bl(\FF_2,[0,\infty,1,0])$ & -- &
   $\phi_2$ & $[0,\infty,1,0]$ & $D_5$ \\
   4& (E) & $Bl(\FF_2,[1,\infty,1])$ & -- &
   $\phi_2$ &$[1,\infty,1]$ & $A_4$ \\
      \hline
       5&$\FF_r$ & $\FF_r$ &$r\ge0$ &
        $\psi_r$ &  $[\,], [\infty]$ & smooth \\
   6& (A) & $Bl(\FF_2,[0],[1])$ & -- &
   $\phi_2^2$
    & $[0],[1]$ & $A_2$ \\
   7& (A) & $Bl(\FF_2,[0,1],[1])$ & -- &
   $\phi_2$ & $[0,1],[1]$ & $A_{3}$ \\
   8& (C) & $Bl(\FF_r,[0,\infty,1,0])$ & $r\ge2$ &
   $\phi_r^1$ & $[\infty],[0,\infty,1,0]$ & $D_4$ \\
   9& (E) & $Bl(\FF_r,[1,\infty,1])$ & $r\ge1$ &
   $\phi_r$ &$[\infty],[1,\infty,1]$ & $A_3$ \\
   \hline
   10& (A) & $Bl(\FF_1,[0],[1])$ & -- &
   $\phi_1$ & $[\infty],[0],[1]$ & $A_1$ \\
   11& (A) & $Bl(\FF_r,[0],[1])$ & $r\ge2$ &
   $\phi_r^2$
    & $[\infty],[0],[1]$ & $A_1$ \\
   12& (A) & $Bl(\FF_r,[0_{(k)},1],[1])$ & $r\ge2$, $ k=$  \newline $1, \ldots, r-1$ &
   $\phi_r$ & $[\infty],[0_{(k)},1],[1]$ & $A_{k+1}$ \\
      \hline 
   \end{tblr}
         \caption{(Abstract) isomorphism classes of projective $\Ga^2$-surfaces with ADE-singularities and finite number of orbits, grouped by the number of orbits.
         The indicated action is the equivariant image of one on $\FF_r$, except for $\phi_r^1$, which indicates a pair of actions if $r\ge3$, and $\phi_r^2$, which indicates a 1-dimensional family of actions. 
         The column ``case'' denotes either the case of the singularity resolution in Fig.~\ref{fig:blowups} or the surface itself.}
      \label{tb:isom-classes}
   \end{table}

\begin{proof}
Assume first that $Y$ is smooth.
Then the only cases without fixed components are $\PP^2$ or $\FF_r$, $r\ge0$.
 By \cite{HaTs99}, there is a unique additive action on $Y$ with a finite number of orbits, which is the image of $\psi_1$ in the former case and $\psi_r$ in the latter case. 
 
Let now $Y$ be singular and $X\to Y$ be its minimal resolution of singularities.
Then the exceptional divisor $E\subset X$ consists of $(-2)$-curves.
In particular, $X$ contains at least one $(-2)$-curve.

By Lemma~\ref{lm:contract-to-Fr-phi}, $X$ is either isomorphic to $(\FF_2,\psi_2)$ or obtained from $(\FF_r,\phi_r)$, where $r\ge1$, by a sequence of blowups of fixed points.
In the former case $Y$ is the weighted projective space $\PP(1,1,2)$ endowed with the image of the action $\psi_2$.
In the latter case Proposition~\ref{pr:blowups}, Corollary~\ref{cr:cases-ACE}, and Proposition~\ref{pr:cases-eq-isom} imply that $Y$ is equivariantly isomorphic to one of surfaces listed in Proposition~\ref{pr:cases-eq-isom}.

The surface $Y$ is obtained from $X$ by contracting all fixed (-2)-curves and possibly the only non-fixed (-2)-curve, which exists in the case $r=2$ and equals the strict transform of the zero section $s$.

Assume first that the exceptional divisor $E\subset X$ 
does not contain a non-fixed curve.
Then, taking into account Lemma~\ref{lm:abs-isom-1-q} for the case (iii), the cases (i), (ii), (iii), and (vi) of Proposition~\ref{pr:cases-eq-isom} lead respectively to rows 10, 11, 12, and 9 of Table~\ref{tb:isom-classes}.
By Lemma~\ref{lm:abs-isom-011}, the cases (iv) and (v) lead to row 8.

Assume now that $E$ contains a non-fixed curve.
Then we have $r=2$ and the cases (ii), (iii), (iv), and (vi)  of Proposition~\ref{pr:cases-eq-isom} lead respectively to rows 6, 7, 3, and 4 of Table~\ref{tb:isom-classes}.

The remaining assertions follow from Table~\ref{tb:isom-classes} and Lemmas~\ref{lm:abs-isom-1-q}--\ref{lm:abs-isom-011}.
\end{proof}

\begin{corollary}\label{cr:inf-orbits}
   Assume that $Y$ is a surface from Table~\ref{tb:isom-classes} distinct from $\PP^2, \PP(1,1,2)$ and $\FF_r$, $r\ge0$. 
   Then $Y$ admits an additive action $\gamma$ with an infinite number of orbits if and only if one of the following cases holds.
   \begin{itemize}
      \item $Y$ is $Bl(\FF_r,[0_{(k)},1],[1])$ with $k\ge r-2$ (no.12), and $(Y,\gamma)$ is equivariantly isomorphic to $Bl(\FF_r,[0_{(k)},0],[1])$;
      \item $Y$ is $Bl(\FF_r,[1,\infty,1])$  with $r\ge 2$ (nos. 4 and 9), and  $(Y,\gamma)$ is equivariantly isomorphic to $Bl(\FF_r,[0,\infty,1])$;
      \item $Y$ is the contraction of the zero section of $X=Bl(\FF_2,[1,\infty,1])$, and $(X,\gamma)$ is again equivariantly isomorphic to $Bl(\FF_r,[0,\infty,1])$.
   \end{itemize}
   In all cases $\gamma$ is the only action with  an infinite number of orbits up to an equivariant automorphism.
   In the first case the action $\gamma$ fixes pointwise the curve $[0_{(k)},1]$, whereas in the second and third cases $\gamma$ fixes the curve $[1,\infty,1]$ or its strict transform respectively.
\end{corollary}
\begin{proof}
   Let $X$ be the minimal singularity resolution of $Y$ and $\sigma_X\colon X\to (\FF_r,\phi_r)$ be the equivariant morphism corresponding to the indicated blowup model of $X$.
   Consider another $\Ga^2$-action $\gamma$ on $X$.
   Then it preserves negative curves of $X$, and by Lemma~\ref{lm:lift-desing} and Corollary~\ref{cr:Fr-model} $\sigma_X$ is $\gamma$-equivariant.

   Since $\gamma$ fixes $f$, it is isomorphic to $\phi_r$ on $\FF_r$. 
   Let $g\in\Aut_f(\FF_r)$ the corresponding isomorphism.
   Then $(X,\gamma)$  is equivariantly isomorphic to $Bl(\FF_r,g\cdot p_1,\ldots,g\cdot p_k)$, where $X=Bl(\FF_r,p_1,\ldots,p_k)$.
In particular, $g\cdot p_j$ are $\phi_r$-fixed points.
   
If all the exceptional components at $g\cdot p_j$ are non-fixed w.r.t. $\phi_r$, then we fall back into one of the cases of Proposition~\ref{pr:cases-eq-isom}.
   Otherwise, there is $p_j$ such that the exceptional component at $p_j$ is not fixed, but the one at $g\cdot p_j$ is fixed.
   By comparing Proposition~\ref{pr:cases-eq-isom} and Lemma~\ref{lm:fixed-curves}, 
   we see that this happens only if either $p=[0_{(k)},1], g\cdot p=[0_{(k)},0]$, where $k\ge1$, or $p=[1,\infty,1], g\cdot p=[0,\infty,1]$.
   Both cases are possible, since an automorphism $g\colon (x,y)\mapsto (x+a_{r-k}y^{r-k},y)$, where $a_{r-k}\in\K$,  sends  $[0_{(k)},0]$ to $[0_{(k)},a_{r-k}]$ and an automorphism $g\colon (x,y)\mapsto (x+a_ry^r,y)$ sends $[0,\infty,1]$ to $[a_r,\infty,w']$ for some $w'\neq0,\infty$.
   The statement is easily seen.
\end{proof}

\begin{remark}\label{rm:moduli}
   Let $X$ be a blowup of $\FF_r$, where $r\ge2$, at two distinct points of the distinguished fiber $f$ not lying on the zero section $s$.
   Then  $X$ is a smooth projective surface, which by Theorem~\ref{th:main} admits a one-dimensional family of pairwise non-isomorphic additive actions.
\end{remark}

\begin{remark}\label{rm:unique-singularity}
   Any normal $\Ga^2$-surface $Y$ with a finite number of orbits contains at most one singular point.
   Indeed, the boundary $D$, which is the complement to the open orbit in $Y$, is connected, since the open orbit is affine.
   Assume that $Y$ contains two singular points. 
   Let $C_1,\ldots,C_k$ be irreducible curves in $D$ that connect them.
   Then each curve $C_i$, $i=1,\ldots,k$, contains at least two nodes of $D$.
   Thus, $C_i$ cannot contain a 1-dimensional $\Ga^2$-orbit, hence it is fixed, a contradiction.
\end{remark}

\begin{remark}
   Given the minimal resolution $X$ of a singular surface $Y$ in Table~\ref{tb:isom-classes}, its exceptional divisor $E$ consists of fixed points if $Y$ is from rows 8--12.
\end{remark}

\bibliographystyle{plainurl}
\bibliography{citations}

\end{document}